\documentclass[a4paper,oneside,11pt]{amsart} 

\usepackage{amsmath, amscd, amsthm, amssymb, mathrsfs, slashed} 
\usepackage[applemac]{inputenc}
\usepackage{geometry}

\usepackage[backend=bibtex8,style=numeric,firstinits,url=false,doi=false,isbn=false,eprint=false]{biblatex}
\renewbibmacro{in:}{%
  \ifentrytype{article}{}{\printtext{\bibstring{in}\intitlepunct}}}
  
\bibliography{gauge_einstein_bibliography.bib}

\DeclareMathOperator{\im}{Im}

\DeclareMathOperator{\End}{End}

\DeclareMathOperator{\Hom}{Hom}

\DeclareMathOperator{\tr}{Tr}

\DeclareMathOperator{\Lie}{Lie}

\DeclareMathOperator{\LC}{LC}
\DeclareMathOperator{\Hor}{Hor}

\DeclareMathOperator{\SO}{SO}

\DeclareMathOperator{\SU}{SU}
\DeclareMathOperator{\GL}{GL}
\DeclareMathOperator{\SL}{SL}

\DeclareMathOperator{\PSU}{PSU}

\DeclareMathOperator{\Ric}{Ric}
\DeclareMathOperator{\vol}{vol}

\newcommand{\R}{\mathbb R}
\newcommand{\C}{\mathbb C}

\newcommand{\G}{\mathcal{G}}

\newcommand{\diff}{\text{\rm d}}
\newcommand{\del}{\partial}

\newcommand{\dirac}{\slashed{\del}}
\newcommand{\dvol}{\mathrm{dvol}}

\newcommand{\so}{\mathfrak{so}}

\renewcommand{\im}{\mathrm {im}\,}
\renewcommand{\P}{\mathbb P}

\theoremstyle{plain}
	\newtheorem{theorem}{Theorem}
	\newtheorem{proposition}[theorem]{Proposition}
	\newtheorem{lemma}[theorem]{Lemma}
	\newtheorem{corollary}[theorem]{Corollary}
	\newtheorem{conjecture}[theorem]{Conjecture}
	\newtheorem{question}[theorem]{Question}
	
\theoremstyle{definition}
	\newtheorem{definition}[theorem]{Definition}
	\newtheorem{remark}[theorem]{Remark}

\theoremstyle{plain}
	\newtheorem*{theorem*}{Theorem}
	\newtheorem*{proposition*}{Proposition}
	\newtheorem*{lemma*}{Lemma}
	\newtheorem*{corollary*}{Corollary}
	\newtheorem*{conjecture*}{Conjecture}
\theoremstyle{definition}
	\newtheorem*{definition*}{Definition}
	\newtheorem*{remark*}{Remark}
	\newtheorem*{remarks*}{Remarks}
	
\makeatletter
\def\blfootnote{\xdef\@thefnmark{}\@footnotetext}
\makeatother

\numberwithin{theorem}{section}

\begin{document}

\title{A gauge theoretic approach to Einstein 4-manifolds}

\begin{abstract}This article investigates a new gauge theoretic approach to Einstein's equations in dimension~4. Whilst aspects of the formalism are already explained in various places in the mathematics and physics literature, our first goal is to give a single coherent account of the theory in purely mathematical language. We then explain why the new approach may have important mathematical applications: the possibility of using the calculus of variations to find Einstein 4-manifolds, as well as links to symplectic topology. We also carry out some of the technical groundwork to attack these problems.
\end{abstract}

\subjclass[2010]{53C25, 53C07, 53D35, 58E30}

\author{Joel Fine} 
\address[Joel Fine]{Département de mathématique\\ 
	Université libre de Bruxelles\\
	Bruxelles 1050\\
	Belgium}
\email{Joel.Fine@ulb.ac.be}
\thanks{JF was supported by an Action de Recherche Concertée and by an Interuniversity Action Poles grant.}

\author{Kirill Krasnov}
\address[Kirill Krasnov]{School of Mathematical Sciences\\
	University Park\\
	Nottingham\\
	NG7 2RD\\
	United Kingdom}
\email{Kirill.Krasnov@nottingham.ac.uk}
\thanks{KK was supported by ERC Starting Grant 277570-DIGT}

\author{Dmitri Panov}
\address[Dmitri Panov]{Department of Mathematics\\
	King's College\\
	Strand\\
	London WC2R 2LS\\
	United Kingdom}
\email{Dmitri.Panov@kcl.ac.uk}
\thanks{DP is a Royal Society Research Fellow}

\date{13th January, 2014. Revised 27th February 2014.}

\maketitle

\section{Introduction}

The focus of this article is a new approach to Einstein 4-manifolds, introduced in the physics literature by the second named author \cite{Krasnov2011Pure-Connection} and independently, albeit in a weaker form focusing on \emph{anti-self-dual} Einstein metrics, by the first named author \cite{Fine2011A-gauge-theoret}. As things stand, the full description of this method is somewhat inaccessible to mathematicians, the necessary background material being spread over several articles written primarily for physicists. Our first goal is to rectify this by giving a single coherent account of the theory in purely mathematical language (\S\ref{Einstein-gauge} below). Our discusion is also more complete in several places than that currently available. 

As we will explain, this formalism potentially has important mathematical applications: it opens up a new way to use the calculus of variations to find Einstein metrics and also has possible applications to symplectic topology. Our second goal is to lay some of the ground work in these directions, ask questions and state some conjectures which we hope will inspire future work (see \S\ref{symplectic_section} and \S\ref{CV_ok}).

\subsection*{Acknowledgements} We would like to thank Claude LeBrun and Misha Verbitsky for helpful discussions, particularly concerning \S\ref{positive_conjetcure_section}. We would also like to thank the anonymous referee of the first draft of this article for suggesting we consider the situation treated in Theorem \ref{no_self_dual_counterexamples}.

\subsection{Main idea}

The key idea is to rewrite Einstein's equations in the language of \emph{gauge theory}, placing them in a similar framework to Yang--Mills theory over 4-manifolds. This reformulation is special to dimension four. We give here the Riemannian version, but one can work equally with Lorentzian signature, replacing $\SO(3)$ throughout by $\SL(2,\C)$, as is explained in \cite{Krasnov2011Pure-Connection}. We give the details in \S\ref{Einstein-gauge}, but put briefly it goes as follows:
\begin{itemize}
\item
Given an $\SO(3)$-connection $A$ over a 4-manifold $X$ which satisfies a certain curvature inequality, we associate a Riemannian metric $g_A$ on $X$, defined algebraically from the curvature of $A$. One can think of $A$ as a ``potential'' for the metric, analogous to the relationship between the electromagnetic potential and field. 
\item
To such connections we also associate an action $S(A) \in \R$, which is simply the volume of the metric $g_A$. Critical points of the action solve a second order PDE which implies that $g_A$ is Einstein. 
\item
There is an a~priori topological bound for $S(A)$ which is attained precisely when $A$ solves a first order PDE (which then implies the second-order equation alluded to above). When this happens the corresponding metric $g_A$ is both anti-self-dual and Einstein, with non-zero scalar curvature.  
\end{itemize}
The parallel with Yang--Mills theory is clear: $S$ plays the rôle of Yang--Mills energy, with Einstein metrics corresponding to Yang--Mills connections and anti-self-dual Einstein metrics being the instantons of the theory. In \S\ref{other_actions}, we briefly discuss other related action principles: the Einstein--Hilbert action, Eddington's action for affine connections and Hitchin's volume functional for stable forms. 

\subsection{Drawbacks}

Before explaining the theory in more detail, it is important to note its current principal failing: the Einstein metrics which arise are precisely those for which $\frac{s}{12}+W^+$ \emph{is a definite endomorphism of $\Lambda^+$}. (Here $s$ is the scalar curvature and $W^+$ is the self-dual Weyl curvature.) We discuss briefly what modifications may be needed to accommodate more general Einstein metrics in \S\ref{extensions}, but there remains much to be done in this direction.

The only possible compact Einstein manifolds for which $\frac{s}{12}+W^+$ is \emph{positive} definite are the standard metrics on $S^4$ and $\C\P^2$ (with the non-complex orientation). This is proved in Theorem~\ref{Einstein_positive_definite}, following an argument which was explained to us by Claude LeBrun. 

When $\frac{s}{12}+W^+$ is negative definite the only \emph{known} compact examples are hyperbolic and complex-hyperbolic metrics (the latter again having the non-complex orientation). Note that, just as for the positive case, here one even has $W^+=0$. It is an interesting open question as to whether these are the only such examples (Question \ref{negative_Einstein_question} below). Some candidate manifolds which have a chance to support such Einstein metrics are described in \S\ref{negative_definite_section}. 

\subsection{Applications}

We believe this reformulation of Einstein's equations will have important  applications. Firstly, it reveals a new link between Einstein 4-manifolds and symplectic Fano and Calabi--Yau 6-manifolds with a certain geometric structure, expanding on that discussed by the first and third named authors \cite{Fine2009Symplectic-Cala}. We explain this in \S\ref{symplectic_section}, where we state Conjecture \ref{positive_definite_conjecture}, which claims that certain 6-dimensional symplectic Fanos are actually algebraic. Proving such a result seems out of reach by current methods in symplectic topology so it is intriguing that the formalism described here suggests a line of attack. Conjecture \ref{positive_definite_conjecture} can also be viewed as a ``gauge theoretic sphere theorem'', saying that only the four-sphere and complex projective plane admit connections whose curvature satisfies a certain inequality. (This might be compared to ``sphere theorems'' in Riemannian geometry, which say that only certain special manifolds admit Riemannian metrics satisfying certain Riemannian curvature inequalities.)

Secondly, the formalism described here gives a new variational approach to Einstein metrics. We lay the groundwork for this in \S\ref{CV_ok}. The traditional action principle for Einstein metrics, via the Einstein--Hilbert action, is not well suited to the calculus of variations, ultimately because its Hessian has an infinite number of eigenvalues of both signs. In our setting however, this problem does not arise: \emph{the Hessian of the volume functional is elliptic with finitely many positive eigenvalues}. This is Theorem~\ref{Hessian_elliptic}. We also briefly discuss the gradient flow of $S$, the analogue of the Yang--Mills flow in this context,  proving short time existence in Theorem \ref{short_time_gf}.

\section{Einstein's equations as a gauge theory}\label{Einstein-gauge}

\subsection{Definite connections as potentials for conformal classes}\label{conformal_class}

We begin with the curvature inequality for an $\SO(3)$-connection mentioned above, and which first appeared in \cite{Fine2009Symplectic-Cala}.

\begin{definition}\label{definite}
A metric connection $A$ in an $\SO(3)$-bundle $E \to X$ over a 4-manifold is called \emph{definite} if whenever $u,v$ are independent tangent vectors, $F_A(u,v) \neq 0$. 
\end{definition}

Plenty of examples of definite connections are given in \cite{Fine2009Symplectic-Cala}. Particularly important to our discussion are those carried by $S^4$, $\C\P^2$, hyperbolic and complex-hyperbolic 4-manifolds. For each of these Riemannian manifolds the Levi-Civita connection on $\Lambda^+$ is definite (where for the complex surfaces we take self-dual forms with respect to the non-complex orientation). 

Given a definite connection $A$ there is a unique conformal class for which $A$ is a self-dual instanton. To see this note that on a 4-manifold a conformal class is determined by the corresponding sub-bundle $\Lambda^+ \subset \Lambda^2$ of self-dual 2-forms. Now, given a local frame $e_1, e_2, e_3$ for $\so(E)$, write $F_A = \sum F_i \otimes e_i$ for a triple of 2-forms $F_i$. We then take the span of the $F_i$ to be $\Lambda^+$. We must check that this sub-bundle satisfies the necessary algebraic condition to be the self-dual 2-forms of some conformal class, namely that the matrix $F_i \wedge F_j$ of volume forms is definite. This, it turns out, is equivalent to $A$ being a definite connection. In what follows we write $\Lambda^+_A$ for the bundle of self-dual 2-forms to emphasise its dependence on $A$. (Notice that this also implicitly orients $X$: given any non-zero $\alpha \in \Lambda^+_A$, the square $\alpha \wedge \alpha$ is positively oriented.)

\subsection{The sign of a definite connection}

Unlike for an arbitrary connection, it is possible to give a sign to the curvature of a definite connection. The starting point is the fact that the Lie algebra $\so(3)$ carries a natural orientation. One way to see this is to begin with two linearly independent vectors $e_1, e_2 \in \so(3)$ and then declare $e_1, e_2, [e_1,e_2]$ to be an oriented basis. One must then check that this orientation does not depend on the initial choice of $e_1, e_2$. Alternatively, and more invariantly, pick an orientation on $\R^3$; the cross product then gives an isomorphism $\R^3 \to \so(3)$ which one can use to push the orientation from $\R^3$ to $\so(3)$. If one begins with the opposite orientation on $\R^3$, the cross-product changes sign and so the resulting orientation on $\so(3)$ is unchanged. 

One consequence of this is that when $X^4$ is given a conformal structure, the resulting bundle $\Lambda^+$ carries a natural orientation. Again, there are various ways to see this. For example, picking a metric in the conformal class, the splitting $\Lambda^2 = \Lambda^+ \oplus \Lambda^-$ corresponds to the Lie algebra isomorphism $\so(4) = \so(3) \oplus \so(3)$, making $\Lambda^+$ into a bundle of $\so(3)$ Lie algebras. Equivalently, given an orthonormal basis $\omega_1, \omega_2, \omega_3$ of $\Lambda^+$ at some point $p$, using the metric to ``raise an index'' on $\sqrt{2}\omega_i$ defines a triple $J_1, J_2, J_3$ of almost complex structures on $T_pX$. There are two possibilities: either the $J_i$ satisfy the quaternion relations, or the $-J_i$ do; the first corresponds to the original choice of basis $\omega_i$ having positive orientation.

\begin{definition}
Let $A$ be a definite connection in an $\SO(3)$-bundle $E \to X$. Let $e_1, e_2, e_3$ be an oriented local frame of $\so(E)$ and write $F_A = \sum F_i \otimes e_i$. $A$ is called \emph{positive definite} if $F_1, F_2, F_3$ is an oriented basis for $\Lambda^+_A$ and \emph{negative definite} otherwise.
\end{definition}

For the standard metrics on $S^4$ and $\overline{\C\P}^2$ the Levi--Civita connections on $\Lambda^+$ are positive definite, whilst hyperbolic and complex-hyperbolic metrics give negative definite connections. 

\subsection{Definite connections as potentials for metrics}

In what follows we fix an orientation on $E$. Together with the fibrewise metric this gives orientation preserving isomorphisms $E^* \cong E \cong \so(E)$ and we will freely identify all these bundles.

We have explained how a definite connection $A$ gives rise to a conformal structure on $X$, the unique one making $A$ into a self-dual instanton.
To specify a metric in this conformal class we need to chose a volume form. 
To see how to do this, pick an arbitrary positively oriented volume form $\nu$. This gives a fibrewise inner-product on $\Lambda^+_A$. Now we can interpret $F_A \in \Lambda^+_A \otimes \so(E) \cong \Hom (E, \Lambda^+_A)$ as an isomorphism $E \to \Lambda^+_A$. Pulling back the inner-product from $\Lambda^+_A$ via this isomorphism gives a new metric in $E$ which differs from the old one by a self-adjoint endomorphism $M_\nu \in \End(E)$. In terms of an orthonormal local trivialisation $e_1, e_2, e_3$ of $E$ in which  $F_A = \sum F_i \otimes e_i$, the matrix representative of $M_\nu$ is determined by 
\[ 
M_{ij} \, \nu = F_i\wedge F_j.
\]
Note that $\nu$ and $M_\nu$ are in inverse proportion, so specifying a volume form is the same as fixing the scale of $M_\nu$. 

A particularly judicious choice of volume form is the following: let $\nu$ be any background choice of (positively oriented) volume form and let $\Lambda$ be any non-zero constant whose sign agrees with that of $A$. Set
\begin{equation}\label{volume_equation}
\nu_A = \frac{1}{\Lambda^2}\left( \tr\sqrt {M_\nu}\right)^2 \nu.
\end{equation}
where $\sqrt{M}$ denotes the positive definite square root of $M$. The homogeneity of (\ref{volume_equation}) means that $\nu_A$ does not depend on the choice of $\nu$. To ease the notation, in what follows we write $M_A$ for $M_{\nu_A}$. An equivalent definition of $\nu_A$ is to demand that if $A$ is positive definite, $\tr\sqrt{M_{A}} =\Lambda$ and if $A$ is negative definite, we have $\tr\sqrt{M_{A}} = - \Lambda$.

\begin{definition}
Given a definite connection $A$, we write $g_A$ for the resulting Riemannian metric with volume form $\nu_A$ defined via (\ref{volume_equation}) and which makes the definite connection $A$ a self-dual instanton.
\end{definition}

The justification for this definition of $\nu_A$ is the following result.

\begin{lemma}
Let $g$ be an Einstein metric, $\Ric = \Lambda g$ and suppose that
$\frac{\Lambda}{3} + W^+$ is a definite endomorphism of $\Lambda^+$. Then the Levi-Civita $\nabla$ connection on $\Lambda^+$ is definite and $g_\nabla = g$.
\end{lemma}
\begin{proof}
Write $\nu$ for the volume form of the Einstein metric. Calculation gives
\begin{equation}\label{M_Riemannian}
M_\nu = \left(\frac{\Lambda}{3} + W^+\right)^2
\end{equation}
When $\frac{\Lambda}{3}+W^+$ is positive definite, $\nabla$ is a positive definite connection.  We can recover the curvature of $\Lambda^+$ via
\[
\frac{\Lambda}{3} + W^+ = \sqrt{M_\nu}
\]
Now $\tr \sqrt{M_\nu} = \Lambda$ is constant. Meanwhile, when $\frac{\Lambda}{3}+W^+$ is negative definite, $\nabla$ is a \emph{negative} definite connection. In this case we have
\[
\frac{\Lambda}{3} + W^+ = - \sqrt{M_\nu}
\]
and $-\tr\sqrt{M_\nu} = \Lambda$ is again equal to the Einstein constant. 
In either case, one sees that $\nu= \nu_\nabla$ and so $g = g_\nabla$.
\end{proof}

%  
%\begin{remark}
%Note that (\ref{M_Riemannian}) shows that given an oriented Einstein 4-manifold, the Levi-Civita connection on $\Lambda^+$ is definite precisely when $\frac{s}{12}+W^+$ is invertible. (When one drops the Einstein assumption one obtains a more general curvature inequality involving the Ricci tensor, given in \cite{ }.) The choice of the \emph{positive definite} square root in the above discusion is what restricts attention to those Einstein metrics for which $\frac{s}{12}+W^+$ is definite. Choosing a different branch of the square root $\sqrt{M_\nu}$ gives a parallel theory which deals with Einstein metrics for which $\frac{s}{12}+W^+$ is invertible but indefinite. In this article we focus purely on the definite case.
%\end{remark}
  
\subsection{Reformulation of the Einstein equations}

We will now give a second order PDE for $A$ which implies that $g_A$ is Einstein. The key is the following well-known observation: an oriented Riemannian 4-manifold $(X,g)$ is Einstein if and only if the Levi--Civita connection on $\Lambda^+$ is a self-dual instanton (see for example \cite{Atiyah1978Self-duality-in}). 

When $A$ is a definite connection in a bundle $E$, it is automatically a self-dual instanton for $g_A$ and, moreover, $E$ is \emph{isomorphic} to $\Lambda^+_A$. To compare $A$ and the Levi-Civita connection on $\Lambda^+_A$ we need an isometry $E \to \Lambda^+_A$. As explained above, the curvature  provides an isomorphism $F_A \colon E \to \Lambda^+_A$ but it is not necessarily isometric, the defect being measured by $M_A \in \End(E)$. To correct for this, define $\Phi_A \colon E \to \Lambda^+_A$ by
\begin{equation}\label{Phi}
\Phi_A = \pm F_A \circ M_A^{-1/2}
\end{equation}
where the sign here agrees with that of the connection $A$. This is an isometry, by definition of $M_A$, which is moreover orientation preserving thanks to the sign.  

\begin{definition}
Given a definite connection  $A$ and $\Phi_A \colon E \to \Lambda^+_A$ defined as in (\ref{Phi}), we write $\LC(A) = \Phi_A^*\nabla$ for the pull-back to $E$ of the $g_A$-Levi-Civita connection in $\Lambda^+_A$.
\end{definition}

In the following result, $\Phi_A$ is treated as a 2-form with values in $E$ (via the identification $\Hom(E, \Lambda^+_A) \cong \Lambda^+_A \otimes E$). The connection $A$ defines a coupled exterior derivative $\diff_A$ on $E$-valued forms and so $\diff_A \Phi_A$ is a 3-form with values in $E$. 

\begin{theorem}[\cite{Krasnov2011Pure-Connection}]
Let $A$ be a definite connection, $\Lambda$ a non-zero constant whose sign agrees with that of $A$ and $\Phi_A$ defined as in (\ref{Phi}). If 
\[
\diff_A \Phi_A = 0,
\] 
then $A = \LC(A)$, $g_A$ is an Einstein metric with $\Ric(g_A) = \Lambda g_A$ and $\frac{\Lambda}{3}+W^+$ definite. Conversely all such  metrics are attained this way.
\end{theorem}

\begin{proof}[Sketch of proof]
If $A = \LC(A)$ then the fact that $A$ is a self-dual instanton implies the Levi-Civita connection is also and hence that $g_{A}$ is Einstein.

We now need a way to recognise $\LC(A)$ amongst all metric connections in $E$, something we can do via \emph{torsion}. Given any metric connection $\nabla$ on $\Lambda^+$, its torsion $\tau(\nabla) \in \Hom(\Lambda^+, \Lambda^3)$ is defined, just as for affine connections, as the difference 
$
\tau(\nabla) = \diff  - \sigma \circ \nabla \colon \Omega^+ \to \Omega^3
$
where $\diff$ is the exterior derivative and $\sigma \colon \Lambda^1 \otimes \Lambda^+ \to \Lambda^3$ is skew-symmetrisation. Paralleling the standard definition of the Levi--Civita connection on the tangent bundle, \emph{the Levi--Civita connection on $\Lambda^+$ is the unique metric connection which is torsion free}. (A proof of this well-known fact can be found in \cite{Fine2011A-gauge-theoret}.) 

We can interpret this from the point of view of $E$. Given a metric connection $B$ in $E$, we push it forward via $\Phi_A$ to a metric connection in $\Lambda^+_A$. The torsion $\tau(\Phi_{A*}(B))$ is identified via $\Phi_A$ with the 3-form $\diff_B \Phi_A$: here the isometry $\Phi_A \colon E \to \Lambda^+_A$ is viewed as an $E^*$-valued 2-form, its coupled exterior derivative is then a section of $\Lambda^3 \otimes E^* \cong \Hom(\Lambda^+_A,\Lambda^3)$ which matches up with $\tau(\Phi_{A*}(B))$. (This calculation is also given in \cite{Fine2011A-gauge-theoret}.) We now see that $\LC(A)$ is the unique metric connection $B$ for which $\diff_B \Phi_A = 0$. 

The upshot of this discussion is that if $\diff_A \Phi_A = 0$ then $A = \LC(A)$ and hence $g_A$ is Einstein. It remains to show that $\Ric(g_A) = \Lambda g_A$ and $\frac{\Lambda}{3}+W^+$ is definite. 

Since the Levi-Civita connection in $\Lambda^+$ is definite, it follows from (\ref{M_Riemannian}) that $\frac{\Lambda}{3}+W^+$ is necessarily invertible. Assume for a contradiction that $\frac{\Lambda}{3}+W^+$ is indefinite, so that $\Lambda^+$ splits into positive and negative eigenbundles. This induces a splitting of $E$ via $\Phi_A$. Now one can check that the equation $\diff_A \Phi_{A} = 0$ forces the sub-bundle of rank one to be $A$-parallel. This implies the curvature of $F_A$ has kernel and so contradicts the fact that $A$ is definite. We also now see from (\ref{M_Riemannian}) that $\Ric(g_A) = \Lambda g_A$, since $\pm \tr \sqrt{M_A} =  \Lambda$, (with sign corresponding to that of $A$).

To prove that all such metrics arise, one simply checks that for an Einstein metric $g$ of the given sort the Levi-Civita connection on $\Lambda^+$ is a definite connection $A$ for which $\Lambda^+_A = \Lambda^+$ is unchanged. Moreover, $\nu_A = \dvol_g$, whilst $\Phi_A \colon \Lambda^+ \to \Lambda^+$ is the identity. Hence $\diff_A \Phi_A=0$ and $g_A = g$.
\end{proof}

An important special case arises when $M_A $ is a multiple of the identity. Since $\tr\sqrt{M_A}$ is constant it follows that $M_A$ is a \emph{constant} multiple of the identity and so $\Phi_A$ is a constant multiple of $F_A$. Now $\diff_A \Phi_A = 0$ follows automatically from the Bianchi identity. Equation (\ref{M_Riemannian}) shows that $g_A$ is in fact anti-self-dual and Einstein, with non-zero scalar curvature. We state this as a separate result:

\begin{theorem}[\cite{Fine2011A-gauge-theoret}]
Let $A$ be a definite connection. If $M_A$ is a multiple of the identity, then $g_A$ is an anti-self-dual Einstein metric with non-zero scalar curvature. Conversely all such metrics are attained this way.
\end{theorem}

\subsection{The action}

A remarkable feature of definite connections and the corresponding metrics is that \emph{the total volume is an action functional for the theory}.

\begin{definition}[\cite{Krasnov2011Pure-Connection}]
Let $A$ be a definite connection in an $\SO(3)$-bundle $E$ over a compact 4-manifold $X$. The \emph{action} of $A$ is defined to be 
\[
S(A) = \frac{\Lambda^2}{12\pi^2}\int_X \nu_A
\]
where $\nu_A$ is defined by (\ref{volume_equation}).
\end{definition}

\begin{theorem}[\cite{Krasnov2011Pure-Connection}]\label{EL_equations}
A definite connection $A$ is a critical point of $S$ if and only if $\diff_A\Phi_A= 0$. So critical points of $S$ give Einstein metrics. 
\end{theorem}
\begin{proof}
Pick a local orthonormal frame $e_1, e_2, e_3$ for $E$ and write $F_A = \sum F_i \otimes e_i$ (where we have as always identified $\so(E) \cong E$ via the cross product). Then  $M_{ij} \nu_A = F_i \wedge F_j$, where $M_{ij} = ( F_i, F_j)$ is the matrix of pointwise innerproducts of the $F_i$ with respect to $g_A$. Making an infinitesimal change $\dot{A} = a$ of the connection gives 
\[
\dot{M}_{ij} \nu_A + M_{ij} \dot{\nu}_A = (\diff_A a)_i \wedge F_j + F_i \wedge (\diff_A a)_j
\]
where $\diff_A a = \sum (\diff_A a)_i \otimes e_i$.  Now the fact that $\tr (\sqrt{M}_A)$ is constant implies that $\tr(M_A^{-1/2}\dot{M}_A) = 0$. Multiply the above equation by $M_A^{-1/2}$ and taking the trace; then use the fact that for any 2-form $\alpha$, $F_i \wedge \alpha = ( F_i , \alpha )\nu_A$ (since the $F_i$ are self dual) to obtain 
\begin{equation}\label{change_nu}
\dot{\nu}_A = \frac{2}{\Lambda} ( \Phi_A, \diff_A a) \nu_A \, .
\end{equation}
Note that the sign in (\ref{Phi}) is correctly captured by the factor of $\Lambda$ here. It follows that 
\[
\dot{S} = \frac{\Lambda}{6\pi^2}\int_X (\diff_A^*\Phi_A, a ) \nu_A
\]
and so the critical points are those $A$ with $\diff_A^*\Phi_A = 0$. On $E$-valued 2-forms, $\diff_A^* = - * \diff_A *$. Since $\Phi_A$ is a self-dual 2-form, $* \Phi_A=\Phi_A$ and so $\diff_A^* \Phi_A = - * \diff_A \Phi_A$. Hence the critical points of $S$ are precisely those $A$ for which $\diff_A \Phi_A = 0$. 
\end{proof}

\subsection{Topological bounds}

We next explain topological bounds on $S$. 

\begin{proposition}[cf.\ \cite{Fine2011A-gauge-theoret}]
There is an a~priori bound 
\[
\frac{1}{3}\left(2\chi(X) + 3\tau(X)\right) < S(A) \leq 2\chi(X) + 3\tau(X)
\] 
for all definite connections. Moreover, $S(A) = 2\chi(X) + 3\tau(X)$ if and only if $M_A$ is a multiple of the identity and hence $g_A$ is anti-self-dual and Einstein with non-zero scalar curvature. 
\end{proposition}

\begin{proof}
Given a self-adjoint positive-definite 3-by-3 matrix $M$, 
\[
\tr(M) < \left(\tr(\sqrt{M})\right)^2 \leq 3 \tr(M)
\]
with equality on the right hand side if and only if $M$ is a multiple of the identity. It follows from (\ref{volume_equation}) that for any choice of volume form $\nu$, $\frac{1}{\Lambda^2}\tr(M_\nu) \nu <  \nu_A \leq \frac{3}{\Lambda^2} \tr(M_\nu) \nu$ with equality on the right if and only if $M_\nu$ (and hence $M_A$), is a multiple of the identity. 

Now $\tr(M_\nu)\nu$ is a multiple of the first Pontryagin form of $A$. To see this, in a local trivialisation in which $F_A = \sum F_i \otimes e_i$, we have $\tr(M_\nu)\nu = \sum F_i \wedge F_i$. Meanwhile, $p_1(A) = \frac{1}{4\pi^2} \sum F_i \wedge F_i$. Hence
\[
\frac{4\pi^2}{\Lambda^2} p_1(A) <  \nu_A \leq \frac{12 \pi^2}{\Lambda^2}p_1(A)
\]
This, together with the fact that $p_1(E) = p_1(\Lambda^+) = 2\chi(X) + 3\tau(X)$,  gives the result.
\end{proof}

As an immediate corollary we see that the existence of a definite connection implies ``one-half'' of the Hitchin--Thorpe inequality satisfied by Einstein 4-manifolds \cite{Hitchin1974Compact-four-di,Thorpe1969Some-remarks-on}:

\begin{corollary}[\cite{Fine2009Symplectic-Cala}]
If $X$ carries a definite connection then $2\chi(X) + 3 \tau(X) >0$.
\end{corollary}

We remark that to date \emph{this is the only known obstruction to the existence of definite connections}. 

\subsection{Potential extensions of the formalism}\label{extensions}

The Einstein metrics which can be reached using the approach given above are precisely those for which $\frac{\Lambda}{3}+W^+$ is a definite endomorphism of $\Lambda^+$. We now discuss briefly how one might extend the formalism to include more Einstein metrics. 

One case that is currently missing is when $(M,g)$ is Einstein and $\frac{\Lambda}{3}+W^+$ is invertible, but not definite. In this case, the Levi-Civita connection $A$ on $\Lambda^+$ is still a definite connection, but $g_A \neq g$. This is because in the above recipe for $g_A$, we fix $\tr(\sqrt{M})=|\Lambda|$, where $\sqrt{M}$ is the \emph{positive} square root of $(\frac{\Lambda}{3}+W^+)^2$. But this is only equal to $\pm(\frac{\Lambda}{3}+W^+)$ when all eigenvalues of the later have the same sign.  The solution is, of course, to choose a different branch of the square root. Using the right choice for $\sqrt{M}$, the normalisation $\tr(\sqrt{M}) = |\Lambda|$ gives $g_A = g$ and the theory proceeds as before. This raises two important questions. Firstly, given a definite connection $A$, how can one tell in advance which branch of the square root to take?  Secondly, one could imagine a situation in which the Levi-Civita connection $A$ of an Einstein metric was, say, positive definite, i.e., that $\det(\frac{\Lambda}{3} + W^+) >0$, and yet $\Lambda <0$. This means that one might not know in advance what sign to take for $\Lambda$. 

The second type of Einstein metric which does not fit directly in the above set-up are those for which eigenvalues of $\frac{\Lambda}{3}+W^+$ vanish somewhere on the manifold. Given such an Einstein metric, with Levi-Civita connection $A$ in $\Lambda^+$, one can proceed as above up to the definition of $\Phi_A = \pm F_A \circ M_A^{-1/2}$, which is not immediately valid where eigenvalues of $M_A$ vanish. One the other hand, since we started with an Einstein metric, one can check that $\Phi_A$ extends smoothly over the seemingly singular locus to a globally defined isometry $\Lambda^+ \to \Lambda^+$. To include such Einstein metrics, we should consider more general connections, which we call \emph{semi-definite}. These are connections for which $M_A \geq 0$ and that $M_A>0$ in at least one point. We demand moreover that $\Phi_A = \pm F_A \circ M_A^{-1/2}$ (defined initially where $M_A>0$) extends smoothly to an isometry $E \to \Lambda^+$ over the whole manifold. 

An example which exhibits some of the problems mentioned above is the Page metric \cite{Page1978A-compact-rotat}. This is an Einstein metric with positive scalar curvature on the one-point blow-up $X$ of $\C\P^2$. It is cohomogeneity-one, being invariant under the $\SU(2)$-action on $X$ lifting that on $\C\P^2$ which fixes the centre of the blow-up. It can be described via a path $g(t)$ of left-invariant metrics on $\SU(2)$ parametrised by $t \in (0,1)$, which collapse the fibres of the Hopf map $\SU(2) \to \C\P^1$ as $t \to 0,1$. The metric $\diff t^2 + g(t)$ on $\SU(2) \times (0,1)$ extends to give a smooth metric on the blow-up $X$. The $\C\P^1$s which compactify the ends of $\SU(2) \times (0,1)$ are the exceptional curve and the line in $\C\P^2$ which is orthogonal to the centre of the blow-up. 

One can check, by direct computation, that the Levi-Civita connection on $\Lambda^+$ of the Page metric is postive definite everywhere except for $\SU(2)\times \{t_0\}$ for a single value of $t_0$. For $t< t_0$, $\frac{\Lambda}{3}+W^+ >0$ and we are in exactly the situation treated in this article. At $t=t_0$, two eigenvalues of $\frac{\Lambda}{3}+W^+$ vanish and for $t >t_0$ these same two eigenvalues become negative. Provided one uses the correct branch of the square root as $t$ crosses $t_0$, the whole Page metric can be seen using the gauge theoretic approach. Indeed one can ``find'' the Page metric by solving the equation $\diff_A \Phi_A = 0$ for an appropriate cohomogeneity-one definite connection. This directly gives the Page metric over $t \in [0,t_0)$ and one sees clearly that two eigenvalues vanish at $t= t_0$. Then taking a different square root gives the metric over $(t_0, 1]$ and it is a simple matter to see that the two parts combine smoothly to give the whole Page metric. 

\section{Related action principles}\label{other_actions}

This section gives some additional context to the above formalism by describing some related action principles.

\subsection{The Einstein--Hilbert action}

The traditional action principle for Einstien metrics is that of Einstein--Hilbert. Given a Riemannian metric $g$, define the Einstein--Hilbert action of $g$ to be
\[
S_{\text{EH}}(g) = \int_X \left( s - 2\Lambda\right) \dvol
\]
where $s$ is the scalar curvature of $g$ and $\Lambda$ is the cosmological constant. Critical points of $S_{\text{EH}}$ are those metrics with $\Ric(g) = \Lambda g$. 

Note that the Einstein--Hilbert functional evaluated on an Einstein metric is a multiple of the volume. By contrast in the gauge-theoretic formulation the action is always the volume of the space, even away from the critical points. 

As is well known, the Einstein--Hilbert action is beset with difficulties from both the mathematical and physical points of view. Physically, the resulting quantum theory is not renormalizable. Mathematically, one hopes to use the calculus of variations to find critical points of an action. For $S_{\text{EH}}$ this is extremely hard: at a critical point the Hessian has infinitely many positive and negative eigenvalues which leads to seemingly unresolvable problems in min-max arguments. As we discuss in \S\ref{CV_ok}, it is precisely these problems that the gauge-theoretic approach outlined above aims to avoid. 

\subsection{Eddington's action for torsion-free affine connections}

The idea of using a connection as the independent variable, rather than the metric, goes back to Eddington \cite{Eddington1921A-Generalisatio}. Let $\nabla$ be a torsion-free affine connection in $TX \to X$. Such connections have a Ricci tensor, defined exactly as for the Levi--Civita connection: the curvature tensor of $\nabla$ is a section of $\Lambda^2 \otimes TX \otimes T^*X$ and so one can contract the $TX$ factor with the $\Lambda^2$ factor to produce a tensor $\Ric_\nabla \in T^*X \otimes T^*X$. 

The fact that $\nabla$ is torsion-free implies that $\Ric_\nabla$ is symmetric. We now focus on those $\nabla$ for which $\Ric_\nabla$ is a definite form.  Pick a non-zero constant $\Lambda$ which is positive if $\Ric_\nabla$ is positive definite and negative if $\Ric_\nabla$ is negative definite. We can then define a Riemannian metric on $X$ by $g = \Lambda^{-1}\Ric_\nabla$. Eddington's action is again the volume:
\[
S_{\text{Edd}}(\nabla) = \int_X \dvol_g
\]
The Euler--Lagrange equations for $S_{\text{Edd}}$ are $\nabla \Ric_\nabla =0$. In other words, $\nabla$ is both torsion-free and makes $g$ parallel, making it the Levi--Civita connection of $g$. So $\Ric_{\nabla} = \Ric_g$ is \emph{the} Ricci curvature and $g$ is Einstein by virtue of its definition as $g =\Lambda^{-1}\Ric_g$.

The parallels with our gauge-theoretic approach are obvious. An important difference, however, is the size of the space of fields. For us the connections are given locally by 12 real-valued functions ($\Lambda^1 \otimes \so(E)$ has rank 12). The relevant gauge group is that of fibrewise linear isometries $E \to E$ (not necessarily covering the identity) which has functional dimension 7 (4 diffeomorphisms plus 3 gauge rotations), leaving the space of definite connections modulo gauge with functional dimension 5. 

On the other hand, torsion-free affine connections on a 4-manifold are given locally by 40 real-valued functions ($S^2T^*X \otimes TX$ has rank 40). Moreover, the gauge group is rather small and consists of just diffeomorphisms which have functional dimension 4. This means that Eddington's theory has a configuration space of functional dimension 36. Finally the traditional theory of metrics modulo diffeomorphisms gives a configuration space of functional dimension 6. This vast increase in the ``number of fields'' makes Eddington's theory even more complicated than the traditional Einstein--Hilbert theory. This is why, in spite of the fact that the action in this formulation is a multiple of the volume, it is unlikely to give a useful variational principle. 

\subsection{Hitchin's functional and stable forms}

The use of total volume as an action has also appeared in the work of Hitchin \cite{Hitchin2000The-geometry-of,Hitchin2001Stable-forms-an}. In these articles, Hitchin studies ``stable forms''. A form $\rho \in \Lambda^p\R^n$ is called stable if its $\GL(n,\R)$-orbit is open in $\Lambda^p\R^n$. A symplectic form on $\R^{2m}$ is one example, but there are also ``exceptional'' examples with $p=3$ and $n=6,7,8$.  A form $\rho$ on a manifold $M$ is called stable if it is stable at each point. The existence of such a form gives a reduction of the structure group of $TM$ to $G \subset \GL(n,\R)$, the stabiliser of $\rho$ at a point. In the three exceptional cases mentioned above, $G$ is $\SL(3,\C)$, $G_2$ and $\PSU(3)$ in dimensions 6, 7 and 8 respectively. In each case $G$ preserves a volume form and so the stable form defines in turn a volume form $\phi(\rho)$ on $M$. 

Hitchin proceeds to study the corresponding volume functional $\rho \mapsto \int_M \phi(\rho)$ restricted those $\rho$ which are closed and lie in a fixed cohomology class. The critical points define interesting geometries. For example, when $p=3, n=6$, a critical stable form defines a complex structure on $M$ with trivial canonical bundle, whilst when $p=3, n=7$, a critical stable form defines a metric on $M$ with holonomy $G_2$ (in particular an Einstein metric). 

There is a formal parallel with the set-up described in \S\ref{Einstein-gauge}. The curvature $F_A$ of a definite connection is modelled on a ``stable'' element of $\Lambda^2(\R^4) \otimes \R^3$, in the sense that its  orbit under the action of $\GL(4,\R) \times \SO(3)$ is open. Moreover, the stabiliser preserves the volume form defined by (\ref{volume_equation}). The condition that the stable form be closed has been replaced by the Bianchi identity $\diff_A F_A = 0$ and the restriction to a fixed cohomology class can be thought of as analogous to considering connections on a fixed bundle: $F_{A+a} = F_A + \diff_Aa + a\wedge a$. The main difference between the situation studied by Hitchin and our set-up is that we work with differential forms taking values in a vector bundle over the manifold (curvature 2-forms), while Hitchin studies ordinary differential forms. Because of this, there are more options for constructing a volume form, leading to a family of actions ``deforming'' the action $S$ considered here. This is discussed in~\cite{Krasnov2011Gravity-as-a-di}.

\section{The symplectic geometry of definite connections}\label{symplectic_section}

\subsection{The symplectic manifold associated to a definite connection}

A definite connection in $E \to X$ is not just a ``potential'' for a Riemannian metric on $X$, it also defines a symplectic form on the unit sphere bundle $\pi \colon Z \to X$ of $E$. The symplectic point of view on definite connections was the original motivation for their introduction in \cite{Fine2009Symplectic-Cala}. We briefly recall the construction here.

Given any (not necessarily definite) connection $A$ we define a closed 2-form $\omega_A$ on $Z$ as follows. Let $V \to Z$ be the vertical tangent bundle (i.e., the sub-bundle $\ker \pi_* \leq TZ$). If we orient the fibres of $E$, then $V$ becomes an oriented $\SO(2)$-bundle, or equivalently a Hermitian complex line bundle. The form $\omega_A$ is the curvature form of a certain unitary connection $\nabla$ in $V$ defined as follows. Given a section of $V$, we differentiate it vertically in $Z$ by using the Levi-Civita connection on the $S^2$-fibres. Meanwhile to differentiate horizontally, we use $A$ to identify nearby fibres of $Z \to X$. Together this defines $\nabla$ and hence the closed form $\omega_A= \frac{i}{2\pi}F_\nabla$. 

\begin{lemma}[\cite{Fine2009Symplectic-Cala}]\label{symplectic_definite}
The form $\omega_A$ is symplectic if and only if $A$ is a definite connection.
\end{lemma}

\begin{remark}
This is inspired in part by twistor theory. The isometric identification $E \to \Lambda^+_A$ gives a diffeomorphism of $Z$ with the twistor space of $(X,g_A)$. However, the twistor almost complex structures are rarely compatible with $\omega_A$. One can check that the Atiyah--Hitchin--Singer almost complex structure $J_+$ is compatible with $\omega_A$ if and only if $g_A$ is anti-self-dual Einstein with positive scalar curvature whilst the Eells--Salamon almost complex structure $J_-$ is compatible with $\omega_A$ if and only if $g_A$ is anti-self-dual Einstein with negative scalar curvature. (See \cite{Atiyah1978Self-duality-in} and \cite{Eells1983Constructions-t} for the definitions and properties of $J_+$ and $J_-$ respectively.)
\end{remark}

\subsection{Symplectic Fano manifolds}

The symplectic geometry of $(Z, \omega_A)$ depends dramatically on the sign of the definite connection. We begin with the following definition (which is not universally standard).

\begin{definition}\label{Fano_CY}
A symplectic manifold $(M, \omega)$ is called a \emph{symplectic Fano} if $[\omega]$ is a positive multiple of $c_1(M, \omega)$. It is called a \emph{symplectic Calabi--Yau} if $c_1(M,\omega) = 0$. 
\end{definition}

\begin{proposition}[\cite{Fine2009Symplectic-Cala}]
If $A$ is a positive definite connection then $(Z, \omega_A)$ is a symplectic Fano. If $A$ is negative definite then $(Z, \omega_A)$ is a symplectic Calabi--Yau.
\end{proposition}

Definition \ref{Fano_CY} is of course directly inspired by the similar terms in use in algebraic geometry. The motivation is that in algebraic geometry the study of Fanos and Calabi--Yaus has special features and one is curious to see to what extent these features extend to symplectic geometry. For example, it is known that in each dimension there is a finite number of deformation families of smooth algebraic Fano varieties (see, e.g., the discussion in \cite{Iskovskikh1977Fano-3-folds-I}). 

The next two results show the current knowledge on the extent to which symplectic and algebraic Fanos diverge:

\begin{theorem}[McDuff \cite{McDuff1990The-structure-o}]
Let $(M,\omega)$ be a symplectic Fano 4-manifold. Then there is a compatible complex structure on $M$ making it a smooth algebraic Fano variety.
\end{theorem}

\begin{theorem}[Fine--Panov \cite{Fine2010Hyperbolic-geom}]
There are symplectic Fanos which are not algebraic, starting at least from dimension 12. 
\end{theorem}

This leads to the following question (to the best of our knowledge, this first appeared in the literature in \cite{Fine2010Hyperbolic-geom}, but had surely been discussed before):

\begin{question}\label{Fano_question}
What is the lowest dimension $2n$ in which all symplectic Fanos are necessarily algebraic? (We see from the above that $2\leq n < 6$.)
\end{question}

\subsection{A gauge theoretic ``sphere'' conjecture for positive definite connections}\label{positive_conjetcure_section}

Since positive definite connections give rise to symplectic Fano 6-manifolds they are useful for probing the $n=3$ case of Question \ref{Fano_question}. In this regard we have the following conjecture, which first appeared in \cite{Fine2009Symplectic-Cala}.

\begin{conjecture}\label{positive_definite_conjecture}
Let $A$ be a positive definite connection over a compact 4-manifold. Then the underlying 4-manifold is either $S^4$ or $\overline{\C\P}^2$ whilst the resulting symplectic Fano is algebraic, being symplectomorphic to the standard structure on either $\C\P^3$ or the complete flag $F(\C^3)$ respectively.
\end{conjecture} 

An initial motivation for this conjecture is the stark contrast between the known examples of positive and negative definite connections. The only known compact positive definite connections are isotopic to the Levi-Civita connections on $\Lambda^+ \to S^4$ or $\Lambda^+ \to \overline{\C\P}^2$. The resulting symplectic 6-manifolds are then $\C\P^3$ or $F(\C^3)$. On the other hand, in the many known examples of negative definite connections over compact 4-manifolds, the resulting symplectic 6-manifold never admits a compatible complex structure. See, for example, the discussions in \cite{Fine2009Symplectic-Cala,Fine2010Hyperbolic-geom,Fine2013The-diversity-o}.

Leaving aside Question \ref{Fano_question}, one can view Conjecture \ref{positive_definite_conjecture} as a gauge theoretic ``sphere'' conjecture. Sphere theorems in Riemannain geometry start with a compact Riemannian manifold whose curvature satisfies a certain inequality and deduce that the underlying manifold is diffeomorphic to a sphere, or spherical space form. There is a long history of such theorems. For a survey of classical results and recent developments, see the article of Wilking \cite{Wilking2007Nonnegatively-a}. A proof of Conjecture \ref{positive_definite_conjecture} would be, to the best of our knowledge, the first example of such a theorem with purely gauge theoretic hypotheses, involving the curvature of an auxialliary connection rather than a Riemannian metric. It would also imply a more traditional sphere-type theorem. To see this we first need the following result from \cite{Fine2009Symplectic-Cala}:

\begin{theorem}[Fine--Panov \cite{Fine2009Symplectic-Cala}]\label{Riemannian_definite}
Let $g$ be a Riemannian metric on an oriented 4-manifold $X$. Write $\Ric_0$ for the trace free Ricci curvature of $g$, interpreted as a map $\Lambda^+ \to \Lambda^-$. The Levi-Civita connection on $\Lambda^+$ is definite if
\begin{equation}\label{Riemannian_ineq}
\left(\frac{s}{12} + W^+\right)^2 > \Ric_0^*\Ric_0
\end{equation}
In this case the sign of the connection agrees with that of $\det \left(s/12+W^+\right)$.
\end{theorem}

With this result in hand, we see that the following purely Riemannian conjecture is implied by the more general Conjecture \ref{positive_definite_conjecture}: 

\begin{conjecture}\label{metric_positive_definite_conjecture}
$S^4$ and $\overline{\C\P}^2$ are the only compact 4-manifolds which admit Riemannian metrics satisfying the curvature inequality (\ref{Riemannian_ineq}) and with $s/12+W^+$ a positive definite endomorphism of $\Lambda^+$.
\end{conjecture}

A small step in the direction of Conjecture \ref{metric_positive_definite_conjecture} was made in \cite{Fine2009Symplectic-Cala}, where the following result appears:

\begin{theorem} [Fine--Panov \cite{Fine2009Symplectic-Cala}].
Let $(M,g)$ be a compact oriented Riemannian 4-manifold which satisfies the curvature inequality (\ref{Riemannian_ineq}) and with $s/12+W^+$ a positive definite endomorphism of $\Lambda^+$. Then $M$ is homeomorphic to the connected sum of $n$ copies of $\overline{\C\P}^2$ for some $n=0,1,2,3$.
\end{theorem}

More evidence for Conjecture \ref{metric_positive_definite_conjecture} is provided by the following result.

\begin{theorem}\label{no_self_dual_counterexamples}
Let $(M,g)$ be a compact oriented Riemannian 4-manifold which satisfies the curvature inequality (\ref{Riemannian_ineq}). If $g$ has positive scalar curvarture and is anti-self-dual then it is conformal to the standard metric on either $S^4$ or $\overline{\C\P}^2$.
\end{theorem}

The proof is based on the following result of Verbitsky. We also give here a shorter variation of Verbitsky's proof. (The result we prove here is Theorem~1.1 of \cite{Verbitsky2012Rational-curves}; see Definition~1.2 of \cite{Verbitsky2012Rational-curves} to pass from the statement there to the language of taming forms.)

\begin{theorem}[Verbitsky \cite{Verbitsky2012Rational-curves}]
Let $(M,g)$ be a compact oriented 4-manifold with anti-self-dual Riemannian metric. If the twistor space $Z$ admits a taming symplectic form then the metric is conformal to the standard metric on either $S^4$ or $\overline{\C\P}^2$.
\end{theorem}
\begin{proof}
We begin by showing that $Z$ is Moishezon, using a result of Campana (Theorem 4.5 of \cite{Campana1991On-twistor-spac}). Let $C$ denote the space of analytic cycles of dimension~1 in $Z$ and write $C_t$ for the irreducible 4-dimensional component which contains the vertical twistor lines. Campana's Theorem says that if $C_t$ is compact then $Z$ is Moishezon. In our case, the complex structure is tamed by a symplectic form and so compactness of $C_t$ follows by Gromov's compactness result. 

Next, we note that if $\omega$ is a taming symplectic form then so is $\omega - \gamma^*\omega$, where $\gamma \colon Z \to Z$ is the real involution on the twistor space. This second form is in a multiple of the anti-canonical class. (This follows from the form of the cohomology of $Z$ which is given by Leray--Hirsch, and is described in, for example, \cite{Hitchin1981Kahlerian-twist}.) So we can assume the taming form represents $c_1(K^{-1})$. As a consequence, for any $k$-dimensional subvariety $V$ of $Z$, $\int_V c_1(K^{-1})^k>0$. We can now apply the Nakai--Moishezon criterion, which holds for Moishezon manifolds (Theorem~6 of \cite{Moishezon1966On-n-dimensiona}, see also Theorem~3.11 of \cite{Kollar1990Projectivity-of}). This tells us that $K^{-1}$ is ample and so $Z$ is a projective manifold. Finally, a theorem of Hitchin \cite{Hitchin1981Kahlerian-twist} guarantees that the standard conformal structures on $S^4$ and $\overline{\C\P}^2$ are the only ones which yield compact Kähler twistor spaces. 
\end{proof}

\begin{proof}[Proof of Theorem \ref{no_self_dual_counterexamples}.]
The inequality (\ref{Riemannian_ineq}) implies that there is a symplectic form on the twistor space. Moreover, it follows from Theorem 4.4 of \cite{Fine2009Symplectic-Cala} that this symplectic form tames the twistor complex structure. The result now follows from Verbitsky's Theorem.
\end{proof}

As further motivation for both believing Conjecture \ref{positive_definite_conjecture} and perhaps seeing how to attack it, we now show it holds for those positive definite connections which are also critical points of the volume function. The following argument was explained to us by Claude LeBrun.

\begin{theorem}\label{Einstein_positive_definite}
Let $A$ be a positive definite connection on a compact 4-manifold, solving the equation $\diff_A \Phi_A = 0$ ensuring that $g_A$ is Einstein. Then $g_A$ is also anti-self-dual. It follows that $g_A$ is the standard metric on either $S^4$ or $\overline{\C\P}^2$ and the corresponding Fano is symplectomorphic to $\C\P^3$ or $F(\C^3)$ respectively.
\end{theorem}

\begin{proof}
Since $A$ is positive definite, the scalar curvature $s=4 \Lambda$ of $g_A$ is positive. If $g_A$ is also anti-self-dual then a theorem of Hitchin \cite{Hitchin1981Kahlerian-twist} ensures that the only possibilities are the standard metrics on $S^4$ and $\overline{\C\P}^2$. The twistor spaces of these 4-manifolds are $\C\P^3$ and $F(\C^3)$ respectively and it is a simple matter to check that the symplectic form $\omega_A$ is the standard Kähler form on these spaces (e.g., for symmetry reasons).

Assume then for a contradiction that $W^+\neq 0$. The key is the following inequality due to Gursky (this first appeared in \cite{Gursky2000Four-manifolds-}, see also the streamlined proof in \cite{Gursky1999On-Einstein-man}). Let $g$ be an Einstein metric with $\Ric=  \Lambda g$, where $\Lambda >0$, on a compact 4-manifold $X$. If $W^+$ is \emph{not} identically zero then
\begin{equation}\label{Gursky_ineq}
\int |W^+|^2 \dvol \geq \frac{2\Lambda^2}{3} \vol(g).
\end{equation}
To exploit this, write $w_1 \leq w_2 \leq w_3$ for the eigenvalues of $W^+$. Now  $w_1+w_2+w_3 =0$ implies $0 \leq w_3  \leq -2w_1$ and $w_1 \leq w_2 \leq -\frac{1}{2}w_1$. Because $w_1<0$ this last two-sided bound implies $|w_2|\leq |w_1|$. Hence $|W^+|^2 = w_1^2 +w^2_2 + w^2_3 \leq 6 w_1^2$. Meanwhile, since $\Lambda/3+W^+$ is positive definite, we see that $w_1^2 < \Lambda^2/9$. Integrating the bounds $|W^+|^2 \leq 6w_1^2 < \frac{2\Lambda^2}{3}$ and applying Gursky's inequality (\ref{Gursky_ineq}) we obtain a contradiction.
\end{proof}

\begin{remark} 
We remark that this result also follows from recent work of Richard and Seshadri \cite{Richard2013Positive-isotro}. They study Riemannian 4-manifolds with positive isotropic curvature on self-dual 2-forms (a weakening of the usual positive isotropic curvature inequality, which is required to hold for all 2-forms and which makes sense in any dimension). They prove that the only such compact Einstein metrics are the standard metrics on $S^4$ and $\overline{\C\P}^2$. They give a different proof of this fact, but as they say it also follows from Gursky's inequality. The argument is almost identical to that given here: having positive isotropic curvature on self-dual 2-forms is equivalent to $W^+<s/6$ and so, for an Einstein metric, $w_3\leq \frac{2\Lambda}{3}$. Reasoning as above then shows that $|W^+|^2< \frac{2\Lambda^2}{3}$ and so by Gursky's inequality the metric must actually be anti-self-dual. 
\end{remark}

We can rephrase Theorem \ref{Einstein_positive_definite} as saying that when $A$ is positive definite and a critical point of the volume functional $S$ then in fact $S(A)$ attains the topological maximum $S(A) = 2\chi + 3\tau$. In particular, for positive definite connections over a compact manifold, intermediate critical points are ruled out.
One possible approach to proving Conjecture \ref{positive_definite_conjecture} would be to use the calculus of variations, or even the gradient flow, to find a maximum of $S$. The goal would be to prove the following result, which, in view of Theorem \ref{Einstein_positive_definite} and Moser's proof of local rigidity of symplectic structures, implies Conjecture \ref{positive_definite_conjecture}.

\begin{conjecture}\label{deformation_positive_definite_conjecture}
Let $A$ be a positive definite connection over a compact 4-manifold. Then $A$ can be smoothly deformed through such connections to a critical point of $S$.
\end{conjecture}

In \S\ref{CV_ok} we lay the basic analytic groundwork for this line of attack: we prove that the Hessian of $S$ is an elliptic second-order operator (modulo gauge) and that its gradient flow exists for short time. 

\subsection{Negative definite connections}\label{negative_definite_section}

Recall that using definite connections, we are able to describe all Einstein metrics for which ${\Lambda}/{3}+W^+$ is definite. Theorem \ref{Einstein_positive_definite} says that when ${\Lambda}/{3}+W^+$ is positive definite, the only compact examples are the standard metrics on $S^4$ and $\overline{\C\P}^2$. In the negative definite case, the only \emph{known} examples are hyperbolic and complex-hyperbolic manifolds. (Again, for the complex surfaces we use the non-complex orientation.)

\begin{question}\label{negative_Einstein_question}
Do there exist compact Einstein 4-manifolds with Einstein constant $\Lambda <0$, for which $\frac{\Lambda}{3}+W^+$ is negative definite, besides hyperbolic and complex-hyperbolic manifolds?
\end{question}

We remark that there is a related well-known problem: do there exist compact anti-self-dual Einstein manifolds with negative scalar curvature, besides hyperbolic and complex-hyperbolic manifolds?

One way to try and produce an Einstein manifold answering Question \ref{negative_Einstein_question} is to start with a negative definite connection and deform it to a critical point of the volume functional. Some possible starting connections are described in \cite{Fine2009Symplectic-Cala}, exploiting a construction of Gromov--Thurston \cite{Gromov1987Pinching-consta}. We recall the idea here. 

Let $M$ be an oriented compact hyperbolic 4-manifold with a totally-geodesic nulhomologous surface $\Sigma \subset M$. Since $\Sigma$ is nulhomologous, for each $m$ there is an $m$-fold branched cover $X_m \to M$, with branch locus $\Sigma$. Pulling back the hyperbolic metric from $M$ gives a singular metric on $X_m$. It can be smoothed in such a way as to give a metric for which (\ref{Riemannian_ineq}) of Theorem \ref{Riemannian_definite} is satisfied, with $s/12+W^+$ a negative definite endomorphism of $\Lambda^+$. It follows that the Levi-Civita connection on $\Lambda^+ \to X_m$ is negative definite. 

Finiteness considerations show that not all $X_m$ achieved this way can admit hyperbolic metrics and in fact Gromov and Thurston conjecture that none of them do. Now one can check that $\tau(X_m)=0$. This follows from the fact that $\Sigma$ is nulhomologous and an equivariant version of the signature theorem (see e.g., equation (15) in the article \cite{Hirzebruch1969The-signature-o} of Hirzebruch.) Since the signature is a multiple of $\int( |W^+|^2 - |W^-|^2)$ any anti-self-dual metric on $X_m$ is automatically conformally flat. It follows that any anti-self-dual Einstein metric on $X_m$ has constant curvature and hence would be hyperbolic. In particular, those non-hyperbolic $X_m$ (conjecturally all of them) do not admit anti-self-dual Einstein metrics. To answer Question \ref{negative_Einstein_question} then, one might start with the negative definite connection on $X_m$ and try to deform it to a critical point of $S$, giving an Einstein metric which is not anti-self-dual.

\section{The Hessian and gradient flow of the volume functional}\label{CV_ok}

In the previous section we discussed finding Einstein metrics by searching for critical points of the volume functional $S$. In this section we lay the basic groundwork for such an approach.

\subsection{The Hessian is elliptic modulo gauge}\label{Hessian_elliptic_section}

We begin with the Hessian of the volume function, $S$, defined on the space of definite connections. Definite connections satisfy a weak inequality and hence form an open set in the space of \emph{all} connections (in, for example, the $C^1$-topology). We use the natural affine structure on the space of connections to take the Hessian of $S$ at $A$, to get a symmetric bilinear form $H$ on $\Omega^1(X,E)$. (Recall we implicitly identify $\so(E) \cong E$ throughout.) Using the $L^2$-innerproduct defined by $g_A$ we can then realise this as 
\[
H(a,b) = \int_X (Da,b)\,\nu_A
\]
for a uniquely determined self-adjoint second-order operator $D$ on $\Omega^1(X,E)$, which we also call the Hessian. Our main result concerning $D$ is:

\begin{theorem}\label{Hessian_elliptic}
The Hessian of $S$ is elliptic modulo gauge, with finitely many positive eigenvalues. 
\end{theorem}

As was mentioned above, this is in sharp contrast to the Hessian of the Einstein--Hilbert action. 

The Hessian in general and this result in particular are far simpler to understand at a definite connection for which $M_A = (\Lambda/3)^2 1_E$ is a multiple of the identity (with the multiple chosen so that $\tr(\sqrt{M_A}) = |\Lambda|$). Recall that for such a connection $g_A$ is anti-self-dual Einstein; meanwhile the functional $S$ attains its a~priori topological maximum (equal to $2\chi(X)+3\tau(X)$) precisely at connections corresponding to ASD Einstein metrics. At such points the functional is clearly concave so must have non-positive Hessian. On a first reading of the proof of Theorem \ref{Hessian_elliptic}, it is useful to focus on this case $M_A = (\Lambda/3)^2 1_E$ where a great many of the formulae and interpretations become simpler. We highlight these simplifications at various points. 

\subsubsection{Gauge fixing}

Before proving Theorem \ref{Hessian_elliptic}, we first make precise what we mean by ``elliptic modulo gauge''. The gauge group $\G$ is the group of all fibrewise linear isometries $E \to E$. This acts on connections by pulling back, preserving both the space of definite connections and the action $S$. Differentiating this action at a connection $A$ gives a map $R_A \colon \Lie(\G) \to \Omega^1(X, E)$.  Since $S$ is gauge-invariant, it follows that its Hessian $D$ vanishes on $\im R_A$. The claim in Theorem \ref{Hessian_elliptic} is that on the orthogonal complement of $\im R_A$, the Hessian is the restriction of an elliptic operator whose spectrum is bounded above.

To describe this orthogonal complement, we need a concrete description of $R_A$. First consider the subgroup $\G_0$ of gauge transformations covering the identity on $X$ (the usual gauge group in Yang--Mills theory). It has Lie algebra $\Lie (\G_0) = C^\infty(X,E)$ and here $R_A$ is given by the familiar formula: $R_A(\xi) = -\diff_A \xi$. 

Next, we use the connection $A$ to determine a vector-space complement to $\Lie(\G_0) \subset \Lie (\G)$ by horizontally lifting vector fields on $X$ to $E$. This gives
\[
\Lie(\G) = \Lie(\G_0) \oplus \Hor_A \cong \Lie(\G_0) \oplus C^\infty(X,TX)
\]
where $\Hor_A \cong C^\infty(X,TX)$ are the horizontal lifts to $E$ of vector fields on $X$. Of course, $\Hor_A$ is not a Lie subalgebra precisely because $A$ has curvature. 

\begin{lemma}\label{linear action of vector}
Given $u \in \Hor_A$ its infinitesimal action at $A$ is $R_A(u) = - \iota_u F_A$.
\end{lemma}
\begin{proof}
We switch to the principal bundle formalism. Let $P \to X$ be the principal frame bundle of $E$. We interpret $A$ as an $\SO(3)$-equivariant 1-form on $P$ with values in $\so(3)$ whilst $\Lie(\G)$ is the Lie algebra of $\SO(3)$-invariant vector fields on $P$. Given \emph{any} element $u \in \Lie (\G)$, the corresponding infinitesimal action on $A$ is $R_A(u) = - L_u(A) = -\diff(A(u)) - \iota_u \diff A$. Now $\Hor_A$ is precisely those $u$ with $A(u) =0$. For such vectors, $\iota_u[A\wedge A] = 2[A(u), A] = 0$. It follows that $\iota_u \diff A = \iota_u F_A$, since $F_A = \diff A + \frac{1}{2}[A\wedge A]$.
\end{proof}

So, given $A$ we have an isomorphism $\Lie(\G) \cong C^\infty(X, E) \oplus C^\infty(X,TX)$ with respect to which the infinitesimal action at $A$ is given by 
\begin{equation}\label{infinitesimal_action}
R_A(\xi, u) = -\diff_A \xi - \iota_u F_A.
\end{equation}
Note that the action is first order in derivatives of $\xi$, but zeroth order in $u$. This means that the gauge fixing will involve a mixture of algebraic and differential conditions. The orthogonal complement to $\im R_A$ is $\ker R_A^*$. Write $f_A \colon TX \to \Lambda^1 \otimes E$ for the bundle homomoprhism $f_A(u) = \iota_u F_A$ and set $W_A = \ker f_A^* = (\im f_A)^\perp$.  Then $R_A^* = -(d_A^* \oplus f_A^*)$ and so $\ker R_A^* = \ker d_A^* \cap W_A$.

\begin{lemma}
There is an $L^2$-orthogonal decomposition,
$
\Omega^1(X, E) = \im R_A \oplus \ker R_A^*
$.
\end{lemma}
\begin{proof}
This follows from combining the orthogonal decomposition $\Omega^1(X,E) = \im \diff_A \oplus \ker \diff_A^*$ provided by elliptic theory and the pointwise orthogonal decomposition $\Lambda^1 \otimes E = \im f_A \oplus \ker f_A^*$.
\end{proof}

By gauge invariance, the Hessian $D$ of $S$ preserves sections of $W_A$. We will prove Theorem \ref{Hessian_elliptic} by showing that on those sections $a$ of $W_A$ for which $\diff_A^*a = 0$, $D$ is equal to a genuinely elliptic second order operator on $C^\infty(W_A)$.

\begin{remark}\label{spin_gauge}
When $M_A$ is a multiple of the identity, the splitting $\Lambda^1 \otimes E = \im f_A \oplus W_A$ coincides with a standard decomposition coming from Riemannian geometry. Write $S_\pm$ for the spin bundles of an oriented Riemannian 4-manifold and $S^m_\pm$ for the $m^\text{th}$ symmetric power of $S_{\pm}$. Whilst $S_{\pm}$ are complex vector bundles, even powers---$S_-^p \otimes S_+^q$ with $p+q$ even---carry real structures whose fixed loci are real vector bundles of half the dimension. In what follows we will only encounter these even powers, which exist globally even when $X$ is not spin. In these cases we use $S_-^p \otimes S_+^q$ to denote the real vector bundle, of real rank $p+q+2$. For example, by $S_- \otimes S_+$ we mean the spin bundle isomorphic to the real cotangent bundle. 

There is an orthogonal decomposition 
\[
\Lambda^1 \otimes \Lambda^+ 
\cong
\Lambda^1 \oplus W,
\]
where $W = S_- \otimes S^3_+$. This follows from the isomorphisms $\Lambda^1 \cong S_+\otimes S_-$ and $\Lambda^+ \cong S^2_+$ and the irreducible decomposition
\[
S_- \otimes S_+ \otimes S^2_+ \cong (S_- \otimes S_+) \oplus (S_- \otimes S^3_+).
\]
Using the identification $\Phi_A \colon E \to \Lambda^+_A$, this splitting pulls back to $\Lambda^1 \otimes E$. One can check that this agrees with $\Lambda^1 \otimes E = \im f_A \oplus W_A$ precisely when $M_A$ is a multiple of the identity (in which case $F_A$ is a multiple of $\Phi_A$). 
\end{remark}

\subsubsection{The principal part of the Hessian}

We next describe the leading order part of the Hessian of $S$. We break the calculation up into a series of small steps. We begin by introducing a first order operator $\delta_A \colon \Omega^1(X,E) \to C^\infty(S^2E)$ defined as follows. Let $e_i$ be a local oriented orthonormal frame for $E$ in which $F_A = \sum F_i \otimes e_i, d_A a= \sum d_A a_i\otimes e_i$. Then
\[
(\delta_A a)_{ij} \, \nu_A = d_A a_i \wedge  F_j + F_i\wedge d_A a_j,
\]
defines the $(i,j)$-components of $\delta_A a$. 

As we will see in the next lemma, $\delta_A$ appears in the first variation of $\nu_A$ and $M_A$. This ultimately leads to the principal part of the Hessian having the form $\delta_A^* \circ L_A \circ \delta_A$ for some endomorphism $L_A$ of the bundle $S^2E$.

\begin{lemma}
Given an infinitesimal change $\dot{A} = a$ in a definite connection, the corresponding changes in $\nu_A$ and $M_A$ are given by
\begin{eqnarray}
\label{dot_nu}\dot{\nu}_A
	&=&
		\frac{1}{|\Lambda|} \tr \left(M_A^{-1/2}\delta_A a\right) \nu_A,\\
\label{dot_M}\dot{M}_A
	&=&
		\delta_A a - 
		\frac{1}{|\Lambda|} \tr \left(M_A^{-1/2}\delta_A a\right) M_A
\end{eqnarray}
\end{lemma}
\begin{proof}
Equation (\ref{dot_nu}) is just a restatement of (\ref{change_nu}) using $\delta_A$. To prove (\ref{dot_M}), use the definition of $M$ in a local orthonormal frame: $M_{ij}\nu_A = F_i \wedge F_j$. Differentiating this gives $\dot{M}_A \nu_A + M_A \dot{\nu}_A = (\delta_A a) \nu_A$ which together with (\ref{dot_nu}) completes the proof.
\end{proof}

\begin{remark}
When $M_A$ is a multiple of the identity, all these formulae simplify. Identifying $E$ and $\Lambda^+$ via $\Phi_A$, we view $\delta_A \colon \Omega^1(X, \Lambda^+) \to C^\infty(S^2\Lambda^+)$. When $M_A = (\Lambda/3)^2 1_E$, this is simply the composition
\[
\Omega^1(X, \Lambda^+) \stackrel{\diff_\nabla}{\longrightarrow}
\Omega^2(X, \Lambda^+) \stackrel{ }{\longrightarrow}
C^\infty(S^2\Lambda^+)
\]
where $\diff_\nabla$ is the covariant exterior derivative associated to the Levi-Civita connection and the second arrow is projection $\Lambda^2 \otimes \Lambda^+ \to S^2\Lambda^+$ onto the symmetric tensor product. 

Notice also that the formulae for $\dot{\nu}_A$ and $\dot{M}_A$ simplify. For example, $\dot{M}_A$ is simply the trace-free part of $\delta_A a$. 
\end{remark}

Another ingredient we will need to describe the Hessian is a formula for the change in ${M_A}^{-1/2}$. This is determined from equation (\ref{dot_M}) for $\dot{M}_A$ by some linear algebra which we now describe.

\begin{lemma}\label{deriv_M-1/2}
Let $M$ be a positive definite matrix and $N$ any symmetric matrix regarded as an infinitesimal change in $M$. The corresponding infinitesimal change in $M^{-1/2}$ is the unique solution $G$ to the equation
\begin{equation}\label{G_equation}
GM^{-1/2} + M^{-1/2}G = - M^{-1}NM^{-1}
\end{equation}
\end{lemma}
\begin{proof}
The equation follows from differentiating $M^{-1/2}M^{-1/2} = M^{-1}$. To see that it has a unique solution, define an endomorphism $f$ of the space of symmetric matrices by 
\begin{equation}\label{H}
H \colon N \mapsto - MNM^{1/2} - M^{1/2}NM.
\end{equation}
Now (\ref{G_equation}) says $H(G) =N$. But one can check $H$ is an isomorphism and so $G$ can be recovered from $N$.
\end{proof}
With this lemma in hand, we define a map $G_A \colon S^2E \to S^2E$ by setting $G_A(N)$ to be the unique solution to (\ref{G_equation}) with $M=M_A$. To give a clean formula for the infinitesimal change in $M_A^{-1/2}$, we  write $L_A \colon S^2E \to S^2E$ for the linear map given by
\begin{equation}\label{L}
L_A(N) = G_A\left(
N -  \frac{1}{|\Lambda|} \tr\left(M_A^{-1/2}N\right)M_A
\right)
\end{equation}

\begin{lemma}\label{MA-1/2}
Given an infinitesimal change $\dot{A} = a$ in a definite connection $A$, the corresponding change in $M_A^{-1/2}$ is $L_A(\delta_A a)$.
\end{lemma}
\begin{proof}
This comes from putting together equation (\ref{dot_M}) for $\dot{M}_A$ with Lemma \ref{deriv_M-1/2}.
\end{proof}

The somewhat complicated looking formula (\ref{L}) for $L_A$ (equivalently the expression (\ref{dot_M}) for $\dot{M}_A$) has a simple geometric explanation. One can check that $G_A$ is self-adjoint and negative definite, i.e., that $\tr\left( G_A(N)N'\right) = \tr \left(N G_A(N')\right)$ and $\tr \left(G_A(N)N\right) < 0$ if $N \neq 0$. (One way is to prove these facts first for $H$ defined in (\ref{H}).) This means we can use $-G_A$ to define a new fibrewise innerproduct on the bundle $S^2E$ by $\langle P, Q \rangle := - \tr\left(G_A(P)Q\right)$. Now $G_A(M_A)=-\frac{1}{2}M_A^{-1/2}$ and so $\langle M_A, M_A \rangle = \frac{|\Lambda|}{2}$. We now see that
\[
N - \frac{1}{|\Lambda|} \tr\left(M_A^{-1/2}N\right)M_A
=
N - \frac{\langle M_A, N\rangle}{\langle M_A, M_A\rangle} M_A
\]
is simply the orthogonal projection of $N$ away from the span of $M_A$ in the new innerproduct $\langle \cdot,\cdot \rangle$. It follows that with respect to the \emph{standard} innerproduct $\tr(PQ)$ on $S^2E$, $L_A$ is self-adjoint, negative semi-definite with kernel spanned by $M_A$. We record this in the following Lemma.

\begin{lemma}\label{L_semi-definite}
The map $L_A \colon S^2E \to S^2E$ defined in (\ref{L}) is a self-adjoint negative semi-definite operator with kernel spanned by $M_A$.
\end{lemma}

\begin{remark}
Again, this all becomes much simpler when $M_A$ is a multiple of the identity. In this case, $G_A(N)=-27(2|\Lambda|^3)^{-1}N$ is just a multiple of $N$ and $L_A(N)$ is $-27(2|\Lambda|^3)^{-1}$ times the projection onto the trace-free part of $N$.
\end{remark}
 
The final ingredient we will need is an explicit formula for $\delta^*_A$.

\begin{lemma}\label{adjoint_delta}
The map $C^\infty(S^2E) \to \Omega^1(X,E)$ given by $N \mapsto -2*\left( \diff_A N \wedge F_A\right)$ is the $L^2$-adjoint of $\delta_A$, where here the symbol $\wedge$ denotes the operation $\left(\Lambda^1 \otimes \End(E)\right) \otimes \left(\Lambda^2 \otimes E\right) \to \Lambda^3 \otimes E$ given by tensoring the wedge product on forms with the natural action of $\End(E)$ on $E$. 
\end{lemma}

\begin{remark} Our convention for the Hodge star is that $\beta \wedge 
*\gamma = (\beta, \gamma)\,\dvol$, which means that the codifferential (on any degree form on a four-manifold) is given by $\diff^* =  - * \diff *$.
\end{remark}
\begin{proof}
First note that, since $\diff_A F_A = 0$, $\diff_A N \wedge F_A = \diff_A (N\cdot F_A)$ where the symbol $\cdot$ here denotes the action of $N$ on the $E$-factor of $F_A \in \Lambda^2 \otimes E$. Next note that since $F_A$ is self-dual, $-2*\diff_A (N \cdot F_A) = 2 \diff_A^*(N \cdot F_A)$. We now see that
\[
\int_X \left(-2*( \diff_A N \wedge F_A), a\right) \nu_A
=
\int_X \left(N \cdot F_A, 2\diff_A a\right) \nu_A.
\]
Write $\diff_A a = \sum  \alpha_i \otimes e_i$, then 
\[
(N \cdot F_A, 2\diff_A a) 
=
2\sum_{i,j,k}\left( N_{ij} F_i \otimes e_j, \alpha_k \otimes e_k\right)
=
2\sum_{i,j} N_{ij} (F_i,\alpha_j)
=
\tr(N \delta_A a)
\]
which completes the proof.
\end{proof}

We can now give a precise description of the leading order part of the Hessian of $S$. As a warm up, and to motivate the general case, we first compute the Hessian of $S$ at a point $A$ for which $M_A$ is a multiple of the identity.

\begin{proposition}\label{instanton_Hessian}
When $M_A$ is a multiple of the identity (so that $g_A$ is anti-self-dual Einstein and $S(A) = 2\chi + 3\tau$ attains a global maximum), the Hessian of $S$ is given by
\begin{equation}\label{hess-instanton-formula}
D = \frac{|\Lambda|}{12\pi^2}\,\delta_A^* \circ L_A \circ \delta_A.
\end{equation}
\end{proposition}
\begin{proof}
Recall from the proof of Theorem \ref{EL_equations} that the derivative of $S$ in the direction $a$ is given by
\[
\diff S(a) = \frac{\Lambda}{6\pi^2}\int_X (\Phi_A, \diff _Aa) \nu_A
\]
where $\Phi_A = \pm F_A \circ M_A^{-1/2}$ (with sign corresponding to that of the definite connection). Writing the integrand in terms of a local frame for $E$, this reads
\[
\diff S(a)
=
\frac{|\Lambda|}{6\pi^2} \int \sum_{i,j} (M_A^{-1/2})_{ij} F_i\wedge d_A a_j.
\]
Differentiating with respect to $A$ in the direction $b$ gives
\[
d^2S(a,b) = \frac{|\Lambda|}{6\pi^2} \int  \left[ \sum_{i,j} L_{ij}(\delta_A b) F_i\wedge d_A a_j + \frac{3}{|\Lambda|} \sum_i \left( d_A b_i \wedge d_A a_i + F_i [b,a]_i\right)\right],
\]
where in the last term we have used the fact that $M_A$ is a multiple of the identity. It is now a matter of simple integration by parts in the last term to see that it vanishes. 
\end{proof}

In the general case, a similar calculation shows that the Hessian has the same principal part as in the anti-self-dual Einstein case. 

\begin{proposition}\label{leading-order_Hessian}
At an arbitrary definite connection $A$, the principal part of the Hessian of $S$ is 
\begin{equation}\label{hess-princ-part}
\hat{D} = \frac{|\Lambda|}{12\pi^2}\,\delta_A^* \circ L_A \circ \delta_A.
\end{equation}
\end{proposition}

\begin{proof}
Recall from the proof of Theorem \ref{EL_equations} that the derivative of $S$ in the direction $a$ is given by
\begin{equation}\label{gradient}
\diff S(a)
=
\frac{\Lambda}{6\pi^2}\int_X (\diff_A^*\Phi_A, a) \nu_A
\end{equation}
where $\Phi_A = \pm F_A \circ M_A^{-1/2}$ (with sign corresponding to that of the definite connection). We now differentiate this with respect to $A$ in the direction $b$. 

The $L^2$-innerproduct here depends on $A$, but the infinitesimal change in metric is \emph{first}-order in $b$. Since we are only interested in the principal part, we can ignore this. The only piece which is second-order in $b$ comes from $\diff_A^*\Phi_A$. In the notation of Lemma \ref{adjoint_delta}, one can check that
\[
\diff_A^* \Phi_A = \pm\frac{1}{2} \delta_A^* \left(M_A^{-1/2}\right).
\]
Again, the dependence of $\delta_A^*$ on $A$ is first order. This means that we need only ultimately consider the linearisation of $A \mapsto M_A^{-1/2}$ in the direction of $b$ (which will then be differentiated again by $\delta_A^*$, becoming second order). The proposition now follows from Lemma \ref{MA-1/2}.
\end{proof}

We next turn to the symbol of $D$. Let $\eta$ be a non-zero covector and write 
\[
\sigma(D,\eta) \colon \Lambda^1 \otimes E \to \Lambda^1 \otimes E
\]
for the symbol of $D$ in the direction $\eta$. 

\begin{proposition}\label{symbol_Hessian}
The symbol $\sigma(D,\eta)$ is negative semi-definite with kernel equal to
\[
\ker \sigma(D,\eta)
=
\{ \eta \otimes e + \iota_uF_A : e \in E,\ u \in TX\}
\]
which is the span of the images of the symbols of $\diff_A$ and $f_A$. (Recall $f_A \colon TX \to \Lambda^1 \otimes E$ is the map $f_A(u) = \iota_uF_A$.) 
\end{proposition}

\begin{proof}
Write $\delta_\eta$ for the symbol of $\delta_A$ in the direction $\eta$. Then by Proposition \ref{leading-order_Hessian}, $\sigma(D,\eta) = (|\Lambda|/12\pi^2) \delta^*_\eta \circ L_A \circ \delta_\eta$ and this is negative semi-definite by Lemma \ref{L_semi-definite}. Moreover, the kernel of $\sigma(D,\eta)$ is precisely those $a$ for which $\delta_\eta(a)$ is a multiple of $M_A$ (possibly zero multiple).  

Since the action is gauge invariant, $D \circ R_A = 0$ or, in other words, $D \circ \diff_A = 0 = D \circ f_A$. This means that the images of the symbols of $\diff_A$ and $f_A$ must lie in the kernel. Alternatively, we can check directly that $\sigma(D,\eta)(\eta \otimes e + \iota_u F_A)=0$. The term $\eta \otimes e$ is killed immediately since $\Sigma_\eta$ begins by wedging with $\eta$. For the second term, we work in a local frame for $E$ in which $F_A = \sum F_ i\otimes e_i$. We have $\Sigma_\eta(\iota_uF_A) = (\eta \wedge \iota_u F_i, F_j) + (\eta \wedge \iota_u F_j, F_i)$. Now  the 5-form $\eta \wedge F_i \wedge F_j$ is necessarily zero. It follows that 
\[
0 
= 
\iota_u(\eta \wedge F_i \wedge F_j)
=
\iota_u \eta \wedge F_i \wedge F_j -
\eta \wedge \iota_u F_i \wedge F_j - 
\eta \wedge F_i \wedge \iota_u F_j
\]
Dividing by the volume form gives that $\delta_\eta(\iota_uF_A) = \eta(u)M_A$ is a multiple of $M_A$.  

It remains to check that we have accounted for the entire kernel, i.e., that it has dimension~7. The map $\delta_\eta \colon \Lambda^1 \otimes E \to S^2E$ is surjective. Indeed it is, up to a factor of 2, the composition of the  surjective operations: wedge with $\eta$ followed by self-dual projection  $\Lambda^1 \otimes E \to \Lambda^+ \otimes E$; identification of $\Lambda^+ \otimes E$ and $E \otimes E$ via $F_A$ and finally projection $E\otimes E \to S^2E$. This means that those $a$ sent to a multiple of $M_A$ by $\delta_\eta$---the kernel of $\sigma(D,\eta)$---form a 7-dimensional subspace of $\Lambda^1 \otimes E$. 
\end{proof}

\begin{proof}[The proof of Theorem \ref{Hessian_elliptic}]

We now prove that $D$ is elliptic modulo gauge, with finitely many positive eigenvalues. Recall that if $\xi \in C^\infty(X, E)$ and $u \in C^\infty(X,TX)$, the infinitesimal gauge action $R_A(\xi,u) = -\diff_A \xi - \iota_uF_A$ is first order in $\xi$ and zeroth order in $u$. To fix for this we will treat the two parts separately. Write $W_A \leq \Lambda^1\otimes E$ for the orthogonal complement of the image map $f_A$ (where $f_A(u) = \iota_uF_A$), i.e., $W_A = \ker f_A^*$. Write $\Pi \colon \Lambda^1 \otimes E \to W_A$ for the orthogonal projection onto $W_A$. Since $D \circ f_A = 0$ and $D$ is self-adjoint we see that $D$ restricts to a map $C^\infty(W_A) \to C^\infty(W_A)$. We will consider the gauge fixed operator $D' \colon C^\infty(W_A) \to C^\infty(W_A)$ given by
\begin{equation}\label{D'}
D'(\Pi a) = D (\Pi a)- \Pi( \diff_A \diff_A^* \Pi a).
\end{equation}
On the orthogonal complement $\ker R_A^* = \ker \diff_A^* \cap W_A$ of the gauge group action, $D'=D$. With Proposition \ref{symbol_Hessian} in hand, it is straight forward to check that $D'$ is genuinely elliptic. Its symbol in the direction of a unit length covector $\eta$ is
\[
\sigma(D',\eta) 
= 
\frac{|\Lambda|}{12\pi^2} \delta_\eta^*L_A\delta_\eta 
- 
\Pi p_\eta \Pi.
\]
where $p_\eta \colon \Lambda^1 \otimes E \to \Lambda^1 \otimes E$ is given by tensoring $1_E$ with the orthogonal projection from $\Lambda^1$ onto the span of $\eta$. 

Suppose that $\sigma(D',\eta)(a)= 0$ for some $a \in W_A$. The symbol of $D'$ is a sum of negative semi-definite operators and so $a$ is in the kernel of both terms. By Proposition \ref{symbol_Hessian}, it follows that $a = \eta \otimes e + f_A(u)$ for some $e \in E$ and $u \in TX$, but now $p_\eta(a) = 0$ forces $e=0$ and hence $a \in \im f_A$, But since $a \in W_A$ is also orthogonal to $\im f_A$, this forces $a =0$ and so $\sigma(D',\eta)$ is negative definite and in particular an isomorphism. 

We have shown that $D'$ is elliptic and so  it has a complete basis of eigenvectors. It follows that the restriction of $D$ to $\ker R_A^* = \ker d_A^*  \cap W_A$ also has a complete basis of eigenvectors (a sub-collection of the eigenvectors of $D'$). It remains to show that $D'$ has finitely many positive eigenvalues, but this follows from the fact that its symbol is negative definite. In our case we can see this directly: replacing $D$ by $\Pi \circ \hat{D} = \frac{|\Lambda|}{12\pi^2}\Pi \delta_A^* L_A \delta_A$ in (\ref{D'}) gives an elliptic operator $\hat{D}'$ (with the same symbol) which is negative semi-definite and so has non-positive spectrum. Now $D$ and $\hat{D}$ can be joined by a smooth path, leaving the principal part alone and interpolating between the lower order pieces. Using this path in place of $D$ in (\ref{D'}) gives a path of elliptic operators (all with the same symbol).  Along this path the spectrum---a discrete subset of $\R$---varies continuously. Since it is non-positive at one end of the path, continuity implies only finitely many eigenvalues can become positive at the other end.
\end{proof}

\subsection{The Hessian as a spin Laplacian} 

When $M_A$ is a multiple of the identity, so that $g_A$ is anti-self-dual Einstein, the action $S(A)$ attains a global maximum and so the Hessian is non-positive. In this section we will show that in fact the Hessian is invertible modulo gauge. This implies that anti-self-dual Einstein metrics with non-zero scalar curvature are locally rigid modulo diffeomorphisms. 

This local rigidity was originally proved for positive scalar curvature by LeBrun \cite{LeBrun1988A-rigidity-theo} and negative scalar curvature by Horan \cite{Horan1996A-rigidity-theo}. In the traditional twistor-theoretic point of view they employ, one uses the Penrose transform to identify infinitesimal anti-self-dual Einstein deformations with the certain twisted harmonic spinors. One must then show that the only such harmonic spinors are zero. Our new formulation of the problem gives an alternative way to interpret the appearance of this Dirac operator: modulo gauge, the Hessian is the corresponding spin Laplacian. 

When discussing gauge-fixing, we defined a sub-bundle $W_A$ of $\Lambda^1 \otimes E$ which was orthogonal to the action $\im f_A$ of tangent vectors. Recall from Remark \ref{spin_gauge} that when $M_A$ is a multiple of the identity, $\Phi_A$ identifies $W_A$ with the twisted spin bundle $S_-\otimes S_+^3 \subset \Lambda^1 \otimes \Lambda^+$. This bundle carries a Dirac operator defined by coupling to the Levi--Civita connection on $S_+^3$, which we denote 
\[
\dirac \colon C^\infty(S_-\otimes S_+^3) \to C^\infty(S_+ \otimes S_+^3)
\]
We will show that $D$ is essentially $\dirac^*\dirac$. To compare these operators, we use $\Phi_A$ to identify $\Lambda^1 \otimes E \to \Lambda^1 \otimes \Lambda^+$. Given $a \in \Lambda^1 \otimes E$, we write $\hat{a}$ for its image in $\Lambda^1 \otimes \Lambda^+$.
 
\begin{proposition}\label{dirac_laplacian}
Let $A$ be a definite connection with $M_A$ a multiple of the identity. Let $a \in \Omega^1(X, E)$ be an infinitesimal deformation of $A$ orthogonal to the gauge group action, i.e., $a \in C^\infty(W_A)$ with $\diff_A^*a=0$. Then 
\[
\widehat{D(a)} =  - c\,\dirac^*\dirac (\hat{a})
\]
for some positive constant $c$, which depends only on $|\Lambda|$.
\end{proposition}

\begin{proof}
Remark (\ref{spin_gauge}) explains that if $a \in W_A$ then  $\hat{a} \in C^\infty(S_-\otimes S^3_+)$ lies in the domain of the Dirac operator. Next, since that $\Phi_A$ identifies $A$ with the Levi-Civita connection $\nabla$, we see that $\diff_\nabla^*\hat{a} = 0$.

Now, equation (\ref{hess-instanton-formula}) for the Hessian of $S(A)$ when $M_A$ is a multiple of the identity shows that the Hessian is, up to multplication by a negative constant factor, the following composition
\begin{equation}\label{hess-asd-sequence}
\Omega^1(X, \Lambda^+) \stackrel{\diff_\nabla}{\longrightarrow} 
\Omega^2(X, \Lambda^+) \stackrel{p}{\longrightarrow}
\Omega^2(X, \Lambda^+) \stackrel{\diff_\nabla^*}{\longrightarrow}
\Omega^1(X, \Lambda^+)
\end{equation}
where $p$ is orthogonal projection from $\Lambda^2 \otimes \Lambda^+$ onto the subspace $S^2_0\Lambda^+ \subset \Lambda^ 2\otimes \Lambda^+$ of symmetric trace-free endomorphisms of $\Lambda^+$.

Meanwhile, the Dirac operator $\dirac$ can also be expressed in terms of $\diff_\nabla$ and $\diff_\nabla^*$. The range of $\dirac$ is $S_+\otimes S^3_+ \cong S^4_+ \oplus S^2_+$. Meanwhile the range of $\diff_\nabla$ is 
\[
\Lambda^2 \otimes \Lambda^+ 
\cong 
(S^2_- \otimes S^2_+) \oplus (S^2_+ \otimes S^2_+)
\cong
(S^2_- \otimes S^2_+) \oplus S^4_+ \oplus \R \oplus S^2_+
\]
and that of $\diff_\nabla^*$ is $\Lambda^+ \cong S^2_+$. The Dirac operator on $S_- \otimes S^3_+$ is built from the $S^4_+$ component of $\diff_\nabla$ and the whole of $\diff_\nabla^*$. 

Now projection of a section of $\Lambda^2 \otimes \Lambda^+$ onto the $S^4_+$ component is precisely the above projection $p$. So when $\diff_\nabla^*\hat{a} = 0$, $\dirac{\hat{a}} = p(\diff_\nabla \hat{a})$. It follows that $\dirac^*\dirac{\hat{a}} = (\diff_\nabla^* \circ p \circ \diff_\nabla) \hat{a}$ which is, up to a negative factor, $\widehat{D(a)}$. 
\end{proof}

\begin{corollary}[cf.\ LeBrun \cite{LeBrun1988A-rigidity-theo} when $\Lambda>0$ and Horan \cite{Horan1996A-rigidity-theo} when $\Lambda <0$]
When $M_A$ is a multiple of the identity, the Hessian of $S$ at $A$ is negative definite modulo gauge. It follows that anti-self-dual Einstein metrics with non-zero scalar curvature are locally rigid modulo diffeomorphism.
\end{corollary}
\begin{proof}
By Proposition \ref{dirac_laplacian} it suffices to show that when the metric is anti-self-dual, Einstein with non-zero scalar curvature, the spin Laplacian $\dirac^*\dirac$ is invertible on sections of $S_-\otimes S^3_+$. This was proved for positive scalar curvature by LeBrun \cite{LeBrun1988A-rigidity-theo} and negative scalar curvature by Horan \cite{Horan1996A-rigidity-theo}.
\end{proof}

At an arbitrary point $A$ (for which $M_A$ is no longer a multiple of the identity) there is a similar story for the principal part (\ref{hess-princ-part}) of the Hessian. First, it is convenient to twist the tangent vector $a \in \Lambda^1 \otimes E$ by $\sqrt{M_A}$. Define the ``twisted tangent'' $\alpha \in \Lambda^1 \otimes E$ by $a = \sqrt{M_A} \alpha$. Explicitly, if $a = \sum a_i \otimes e_i$, then 
\[
\alpha = \sum (M_A^{-1/2})_{ij}a_j \otimes e_i.
\]
One then finds that the story for the principal part of the Hessian on the space of twisted tangents $\alpha$ repeats the anti-self-dual Einstein case almost exactly. The only additional difference is that in this case instead of the projection $p$ on the trace-free part one has to insert a more complicated operator on $S^2_0\Lambda^+ \cong S_+^4$ into (\ref{hess-asd-sequence})  that depends non-trivially on $M_A$. One then finds that the principal part of the gauge-fixed operator can be written in a form analogous to Proposition \ref{dirac_laplacian}, but with an additional endomorphism $Q \colon  S_+\otimes S_+^3 \to S_+\otimes S_+^3$ (defined in terms of $M_A$) inserted:
\[
C^\infty(S_- \otimes S_+^3) \stackrel{\dirac}{\longrightarrow} 
C^\infty(S_+\otimes S_+^3) \stackrel{Q}{\longrightarrow} 
C^\infty(S_+\otimes S_+^3) \stackrel{Q^*}{\longrightarrow} 
C^\infty(S_+\otimes S_+^3)  \stackrel{\dirac^*}{\rightarrow}  
C^\infty(S_- \otimes S_+^3)
\]
This shows that in general the principal part of the gauge-fixed Hessian is still given by the square of the Dirac operator, but with a non-trivial ``twisting matrix'' $Q^* Q$ inserted. The main complication in the general case is that there are also other terms, apart from the principal part, which are absent in the anti-self-dual Einstein case. 

\subsection{Short time existence of the gradient flow of $S$}\label{gf_section}

We now turn to the gradient flow of $S$, which is given by
\begin{equation}\label{gradient_flow}
\frac{\del A}{\del t}
=
\frac{\Lambda}{6\pi^2}\diff_A^* \Phi_A.
\end{equation}
(The actual constant factor appearing here is unimportant as long as the sign agrees with that of the definite connection.) This is the analogue of the Yang--Mills flow in our theory. 

The fact that the Hessian of $S$ is elliptic modulo gauge implies that the gradient flow (\ref{gradient_flow}) is parabolic modulo gauge and hence exists for short time. This is standard in the study of geometric flows, going by the name of ``de~Turck's trick'' in the study of Ricci flow \cite{DeTurck1983Deforming-metri}; see also \cite{Donaldson1990The-geometry-of} for a similar proof of the short-time existence of Yang--Mills flow. Accordingly we are content to merely sketch the details.

\begin{theorem}\label{short_time_gf}
Let $A_0$ be a definite connection. Then there is a  short-time solution $A(t)$ to the gradient flow (\ref{gradient_flow}) with $A(t) = A_0$.
\end{theorem}

\begin{proof}[Sketch of proof]
We will first consider a gauge-equivalent flow:
\begin{equation}\label{gauge_fixed_gf}
\frac{\del A}{\del t} 
= 
\frac{\Lambda}{6\pi^2} \diff_A^*\Phi_A
-
\diff_A \diff_A^* (A-A_0)
-
f_A(u_A).
\end{equation}
Here $u_A \in C^\infty(X,TX)$ is the unique vector field which solves the equation
\begin{equation}\label{u_A}
f_A(u_A) = \left(1- \Pi \right)\left(
\frac{\Lambda}{6\pi^2} \diff_{\hat{A}}^*\Phi_{\hat{A}}
-
\diff_{\hat{A}} \diff_{\hat{A}}^* (A-A_0)
\right)
\end{equation}
(Recall that $\Pi$ is the pointwise orthogonal projection of $\Lambda^1 
\otimes E$ onto $W_A$ and so $1-\Pi$ is projection onto $\im f_A$.) By definition of $u_A$ this flow is tangent to $W_A$. We also ask that $A(0)= A_0$.

At $t=0$, the principal part of the linearisation of this flow is is the operator $D'$ on $C^\infty(W_A)$ of (\ref{D'}). To see this note that the first two terms in (\ref{gauge_fixed_gf}) contribute $D - \diff_A\diff_A^*$ (since at $t=0$ any effect on linearising $\diff_A\diff_A^*$ is then evaluated on $A(0)-A_0 = 0$). Moreover, $f_A$ is first order in $A$ and so the contribution to the principal part coming from $f_A(u_A)$ is determined purely by the change in $u_A$. From the definition (\ref{u_A}) of $u_A$ this cancels exactly $(1-\Pi)(D - \diff_A \diff_A^*)$, leaving $D' = D - \Pi \diff_A\diff_A^*$ as the principal part as claimed.

As we have seen, the symbol of $D'$ is negative definite. It follows from the standard theory of parabolic PDE that (\ref{gauge_fixed_gf}) has a  solution starting, which exists at least for short time. Write $A'$ for the solution to (\ref{gauge_fixed_gf}) with $A'(0)=A_0$. We now consider the path of gauge transforms $g(t) \in \G$ which are generated by $-\diff_{A'}^* (A'-A_0) - u_{A'}$, with $g(0)$ the identity. (Here, we have implicitly used $A'$ to horizontally lift $u_{A'}$ and so consider it as an infinitesimal gauge transformation). By gauge invariance of $\diff_A^*\Phi_A$ and the formula (\ref{infinitesimal_action}) for the infinitesimal action of $-\diff_{A'}^* (A'-A_0) - u_{A'}$, it follows that $A(t) = g(t)^*A'(t)$ solves the original gradient flow (\ref{gradient_flow}) with $A(0) =A_0$.
\end{proof}

\printbibliography

\end{document}